\documentclass[a4paper,twoside,12pt]{article}

\usepackage{amssymb,amsmath,amsthm,latexsym}
\usepackage{amsfonts}
\usepackage{graphicx,caption,subcaption}
\usepackage[pdftex,bookmarks,colorlinks=false]{hyperref}
\usepackage[hmargin=1.2in,vmargin=1.2in]{geometry}
\usepackage[mathscr]{euscript}
\usepackage[auth-lg,affil-sl]{authblk}
\usepackage{float}

\newcommand{\noi}{\noindent}
\newcommand{\N}{\mathbb{N}}

\newcommand{\J}{\mathscr{J}}

\newtheorem{theorem}{Theorem}[section]
\newtheorem{definition}[theorem]{Definition}
\newtheorem{lemma}[theorem]{Lemma}
\newtheorem{proposition}[theorem]{Proposition}
\newtheorem{corollary}[theorem]{Corollary}

\title{\textbf{\sc Rainbow Neighbourhood Equate Number of Graphs}}

\author{Johan Kok$^1$, Sudev Naduvath$^2$\footnote{Corresponding Author}}
\affil{\small Centre for Studies in Discrete Mathematics\\ Vidya Academy of Science \& Technology \\ Thrissur, Kerala, India.\\{\tt $^1$kokkiek2@tshwane.gov.za},\\{\tt $^2$sudevnk@gmail.com}}

\date{}

\begin{document}
\maketitle
	
\begin{abstract}
In this paper, a new invariant of a graph namely, the rainbow neighbourhood equate number of a graph $G$ denoted by $ren(G)$ is introduced. It is defined to be the minimum number of vertices whose removal results in a subgraph that admits a $J$-colouring. The new notions of chromatic degree of a vertex $d_\chi(v)$, the maximum and minimum chromatic degrees of $G$ denoted, $\Delta_\chi(G)$ and $\delta_\chi(G)$ respectively, are also introduced. The chromatic diameter of $G$ denoted, $d(G,\chi)$ is introduced as well. The study of $ren(G)$ appears to be very complex for graphs in general so for now, only introductory results will be presented. Finally, the concept of a chromatic degree sequence is proposed as a new research direction. 
\end{abstract}

\noi\textbf{Keywords:} $J$-colouring, rainbow neighbourhood, rainbow neighbourhood equate number, chromatic degree, chromatic diameter.

\vspace{0.25cm}

\noi \textbf{Mathematics Subject Classification 2010:} 05C15, 05C38, 05C75, 05C85. 

\section{Introduction}

For general notation and concepts in graphs and digraphs see \cite{1,5,11}. The \textit{order} of a graph $G$ is the number of vertices in $G$ and is denoted by $\nu(G)=n \geq 1$ and the \textit{size} of $G$ is the number of edges in $G$ and is denoted by $\varepsilon(G)= p \geq 0$. The \textit{degree} of a vertex $v \in V(G)$ is denoted $d_G(v)$ or when the context is clear, simply as $d(v)$. We also say the distance between two vertices $u,v$ denoted, $d(u,v)$ is the length of the shortest $uv$-path in $G$. Unless mentioned otherwise, all graphs $G$ are simple, connected and finite graphs.

We recall that a set $\mathcal{C}$ of distinct colours, say $\mathcal{C}= \{c_1,c_2,c_3,\ldots,c_\ell\}$, $\ell$ sufficiently large, is called a \textit{proper vertex colouring} of a graph $G$, denoted $\varphi:V(G) \mapsto \mathcal{C}$, if no two distinct adjacent vertices have the same colour. The cardinality of a minimum set of colours which allows a proper vertex colouring of $G$ is called the \textit{chromatic number} of $G$ and is denoted $\chi(G)$. We call such a colouring a \textit{$\chi$-colouring} of $G$. 

The number of times a colour $c_i$ is allocated to vertices of a graph $G$ is denoted by $\theta(c_i)$ and $\varphi:v_i \mapsto c_j$ is abbreviated as $c(v_i)=c_j$. When a vertex colouring is considered with colours of minimum subscripts the colouring is called a \textit{minimum parameter colouring}. Unless stated otherwise, we consider minimum parameter colour sets throughout this paper. In addition to this, for coloring the vertices of a graph, we also follow a convention called the \textit{rainbow neighbourhood convention} (see \cite{6}) with respect to which we always colour the maximum possible number of vertices with the colour $c_1$, followed by maximum possible number of remaining vertices with colour $c_2$, and proceeding like this, in the final step we color the remaining uncolored vertices with color $c_\ell$. Such a colouring is called a \textit{$\chi^-$-colouring} of a graph or simply a chromatic colouring. We also say that a graph is chromatically coloured. 

Recall that any graph $G$ of order $\nu(G) \geq 2$ has at least two vertices of equal degree. Various problems related to the aforesaid have been studied. In \cite{3}, a formal study on the repetition number of a graph has been conducted. Essentially the maximum multiplicity of elements in the degree sequence of a graph was studied. In \cite{2}, the largest induced subgraph $H$ in which there exist at least $k$ vertices realizing the maximum degree, $\Delta(H)$ has been reported. The aforesaid was a problem posed in \cite{4}. In essence, it means the removal of the minimum number of vertices from $G$ to ensure the condition. Motivated by this the notion of the rainbow neighbourhood equate number of a graph $G$, denoted by $ren(G)$, will be discussed in the following section.


The closed neighbourhood $N[v]$ of a vertex $v \in V(G)$ which contains at least one coloured vertex of each colour in the chromatic colouring, is called a rainbow neighbourhood. We say that vertex $v$ yields a rainbow neighbourhood. We also say that vertex $u \in N[v]$ belongs to the rainbow neighbourhood $N[v]$.

\begin{definition}{\rm \cite{10}
A maximal proper colouring of a graph $G$ is said to be a \textit{Johan colouring} or \textit{$J$-colouring} of $G$, if and only if every vertex of $G$ belongs to a rainbow neighbourhood of $G$. The maximum number of colours in a $J$-colouring is denoted by $\J(G)$.
}\end{definition}

\begin{definition}{\rm $[10]$
A maximal proper colouring of a graph $G$ is a \textit{modified Johan colouring} or a \textit{$J^*$-colouring}, if and only if every internal vertex of $G$ belongs to a rainbow neighbourhood of $G$. The maximum number of colours in a $J^*$-colouring is denoted by $\J^*(G)$.
}\end{definition}

Note that, if a graph $G$ admits a $J$-colouring it admits a $J^*$-colouring. However, the converse is not always true. Many other interesting results on $J$-colouring are found in \cite{7, 8, 9, 10}. We recall some of the important results provided in \cite{10}.

\begin{proposition}\label{Prop-1.3}{\rm \cite{10}}
Let $P_n$ denotes a path on $n\geq 3$ vertices. Then, $\J(P_n)=2$ and $\J^*(P_n)=3$.
\end{proposition}

\begin{theorem}\label{Thm-1.4}{\rm \cite{10}}
A cycle $C_n$ is $J$-colourable if and only if $n\equiv 0\,({\rm mod}\ 2)$ or $n\equiv 0\,({\rm mod}\ 3)$.
\end{theorem}

\begin{corollary}\label{Cor-1.5}{\rm \cite{10}}
Let $C_n$ be a cycle which admits a $J$-colouring. Then, 
\begin{equation*}
\J(C_n)=
\begin{cases}
2 & \text{if}\ n\equiv 0\,({\rm mod}\ 2) \text{and}\ n\not\equiv 0\,({\rm mod}\ 3),\\
3 & \text{if}\ n\equiv 0\,({\rm mod}\ 3).
\end{cases}
\end{equation*}
\end{corollary}

\begin{theorem}\label{Thm-1.6}{\rm \cite{10}}
Let $W_{n+1}=K_1+C_n$ be a wheel graph which admits a $J$-colouring. Then, 
\begin{equation*}
\J(W_{n+1})=
\begin{cases}
3 & \text{if}\ n\equiv 0\,({\rm mod}\ 2) \text{and}\ n\not\equiv 0\,({\rm mod}\ 3),\\
4 & \text{if}\ n\equiv 0\,({\rm mod}\ 3).
\end{cases}
\end{equation*}
\end{theorem}

\section{The Rainbow Neighbourhood Equate Number of a Graph}

In general, not all graphs admit a $J$-colouring. The question here is to find a maximal induced subgraph $H$ of $G$ such that $H$ admits a $J$-colouring. In view of this problem, we have the following definition.

\begin{definition}{\rm 
Let $G$ be a graph. Then, the \textit{rainbow neighbourhood equate number} of $G$ denoted by $ren(G)$ is the minimum number of vertices to be removed from $G$ such that the remaining induced subgraph $H$ admits a $J$-colouring.
}\end{definition}

That is, if the removal of the minimum number of vertices, say $k$, from $G$ such that the remaining induced subgraph $H$ of $G$ admits a $J$-colouring, then $ren(G)=k$. Note that for a graph $G$ which admits a $J$-colouring, $ren(G)=0$.

The existence of a maximal induced subgraph that admits a $J$-colouring is indeed guaranteed by the next result. 

\begin{theorem}\label{Thm-2.1}
For any graph $G$ a maximum induced subgraph $H$ exists such that $H$ admits a $J$-colouring.
\end{theorem}
\begin{proof}
Since $G$ is connected it has $\varepsilon(G) \geq 1$ so consider any edge $uv$. The induced subgraph $H' =\langle\{u,v\}\rangle$ admits the $J$-colouring $\mathcal{C} = \{c_1,c_2\}$. Since $H'$ admits, it follows by implication that a maximal induced subgraph $H$, possibly $H = H'$, that admits a $J$-colouring exists.
\end{proof}

In view of Theorem \ref{Thm-2.1}, for any graph $G$, we have $ren(G)\leq \nu(G)-2=n-2$. Furthermore, the theorem only implies that $\chi(H)\leq \chi(G)$ and not necessarily that $\chi(H)<\chi(G)$.

\begin{lemma}\label{Lem-2.3}
The fan graph $F_n=P_n+K_1$ has $\J(F_n)=3$.
\end{lemma}
\begin{proof}
The result follows immediately from Proposition \ref{Prop-1.3}.
\end{proof}

In view of Proposition \ref{Prop-1.3} and Theorem \ref{Thm-1.4} and Theorem \ref{Thm-1.6}, the next results follow easily.

\begin{corollary}\label{Cor-2.4}
\begin{enumerate}\itemsep0mm
\item[(a)] Let $P_n$ denotes a path on $n\geq 3$ vertices. Then, $ren(P_n) = 0$.
\item[(b)] Let $C_n$ be a cycle with $n \not \equiv 0\,({\rm mod}\ 2)$ and $n \not \equiv 0\,({\rm mod}\ 3)$. Then, $ren(C_n) = 1$, else $ren(C_n) = 0$.
\item[(c)] Let $W_{n+1}$ be a wheel graph with $n \not \equiv 0\,({\rm mod}\ 2)$ and $n \not \equiv 0\,({\rm mod}\ 3)$. Then, $ren(W_{n+1}) = 1$, else $ren(W_{n+1}) = 0$.
\end{enumerate}
\end{corollary}
\begin{proof}
\begin{enumerate}\itemsep0mm
\item[(a)] Since $P_n$ admits a $J$-colouring for all $n \geq 1$, the result is trivial.
\item[(b)] Since $C_n,\ n \not \equiv 0\,({\rm mod}\ 2)$ and $n \not \equiv 0\,({\rm mod}\ 3)$ has order $n \geq 3$ the removal of any vertex results in a path $P_{n-1}$. Hence, the result from part (a). Else, $C_n$ admits a $J$-colouring.
\item[(c)] Since $W_{n+1},\ n \not \equiv 0\,({\rm mod}\ 2)$ and $n \not \equiv 0\,({\rm mod}\ 3)$ has order $n \geq 4$ the removal of any cycle vertex results in a fan graph $F_{n-1}$ hence, the result from Lemma \ref{Lem-2.3}. Else, $W_{n+1}$ admits a $J$-colouring.
\end{enumerate}
\end{proof}

\noi We recall an important result from \cite{8}.

\begin{theorem}\label{Thm-2.5}{\rm \cite{8}}
Irrespective of whether a graph $G$ admits a $J$-colouring or not , its Mycielskian graph $\mu(G)$ does not have a $J$-colouring.
\end{theorem}

\begin{corollary}\label{Thm-2.6}
If a graph $G$ admits a $J$-colouring, its Mycielskian graph has, $ren(\mu(G))=1$.
\end{corollary}
\begin{proof}
When we remove the root vertex of $\mu(G)$, the graph will be reduced to the shadow graph $s(G)$ of a $G$. Clearly, $s(G)$ admits a $J$-colouring with respect to the same colouring of $G$. Therefore, $ren(\mu(G))=1$.
\end{proof}

Note that the result follows directly from the proof of Theorem 2.8 in [8]. Now, recall the definition of a Jahangir graph.

\begin{definition}{\rm \cite{9}
A \textit{Jahangir graph} $J_{n,m},\ n\geq 1$ and $m \geq 3$, is a graph on $nm+1$ vertices consisting of the cycle $C_{nm}$ with an additional central vertex say $u$ which is adjacent to cyclically labeled vertices $v_1, v_2, v_3,\ldots,v_m$ such that $d(v_i,v_{i+1})=n,\ 1\leq i \leq m-1$ in $C_{nm}$.
}\end{definition}

The following result discusses the rainbow neighbourhood equate number of Jahangir graphs.

\begin{theorem}
If a Jahangir graph $J_{n,m}$ does not admit a $J$-colouring, then we have
\begin{equation*}
ren(J_{n,m})=
\begin{cases}
1, & \text{if $C_{nm}$ admits a  $J$-colouring};\\
2, & \text{otherwise}.
\end{cases}
\end{equation*}
\end{theorem}
\begin{proof}
\begin{enumerate}\itemsep0mm
\item[(i)] If $C_{mn}$ admits a $J$-colouring, then since $J_{n,m}-u=C_{nm}$, the result is immediate.
\item[(ii)] If $C_{mn}$ admits a $J$-colouring, then for any $i$, $J_{n,m}-\{v_{i},u\}=P_{nm-1}$. Hence, the result is immediate.
\end{enumerate}
\end{proof}

\section{Some General Results}

In this section, we determine the rainbow neighbourhood equate number of two operations between the graphs $G$ and $H$. First, consider the join of two graphs.  

\begin{theorem}
Consider graphs $G$ and $H$. Then, $ren(G+H)=ren(G)+ren(H)$.
\end{theorem}
\begin{proof}
Let $\chi(G)=k$ and $\chi(H)=\ell$. Without loss of generality, assume that a $\chi$-colouring of $G+H$ corresponds to $c(v) \in \{c_1,c_2,c_3,\ldots,c_k\},\ v \in V(G)$, and $c(u) \in \{c_{k+1}, c_{k+2}, c_{k+3}, \ldots, c_{k+\ell} \},\ u \in V(H)$. Let vertex $v'\in V(G)$ be a vertex to be removed amongst those required to result in an induced subgraph $G'$ which admits a $J$-colouring, if such vertex exists. In the graph $G+H$ vertex $v'$ does not (cannot) yield a rainbow neighbourhood. Hence, one of the minimum number of vertices is to be removed from $V(G)$. Similar reasoning applies to say vertex $u'\in V(H)$. Therefore, $ren(G+H)=ren(G)+ren(H)$.
\end{proof}

\noi Next to be considered is the corona of two graphs. First, we introduce a new distance measure as follows:
 
\begin{definition}\label{Def-3.1}{\rm 
Let $\mathcal{C}$ be a chromatic colouring of a graph $G$. The number of distinct colours in $N[v],\ v \in V(G)$ is called the \textit{chromatic degree} of $v$, and is denoted by $d_\chi(v)$. The maximum and minimum chromatic degrees of $G$ are defined as, $\Delta_\chi(G)=\max\{d_\chi(v): v\in V(G)\}$ and $\delta_\chi(G)=\min\{d_\chi(v): v\in V(G)\}$.
}\end{definition}

\begin{definition}{\rm 
The \textit{chromatic diameter} of graph $G$, denoted by $d(G,\chi)$, is defined as $d(G,\chi) = \chi(G)-\delta_\chi(G)$.
}\end{definition}

As much as the result for the join $G + H$ is straight forward the corona serves to show how complex graph operations can render the problem in general. Note that three parameters are key to the consideration of the corona operation. For the two graphs $G$ and $H$ in $G\circ H$ the parameters $\chi(G)$, $\chi(H)$, order of graphs $G$, $H$, and whether or $G$ and/or $H$ are $J$-colourable are all relevant. It is obvious that $ren(K_1\circ K_1)=0$.

\begin{lemma}\label{Lem-3.4}
Consider graphs $G$ and $H$ of order $n \geq 2$ and $m \geq 2$, and $\chi(G) = k$, $\chi(H) = \ell$, respectively. Also let both graphs be chromatically coloured. Then: 
\begin{equation*}
ren(G\circ H)=
\begin{cases}
0, & \text{if $d(H,\chi)=0$ and $\ell\ge k-1$},\\
(m+1)\cdot\sum\limits_{i=\ell +2}^{k}\theta(c_i), & \text{if $d(H,\chi)=0$ and $\ell<k-1$}.
\end{cases}
\end{equation*}
\end{lemma}
\begin{proof}
\begin{enumerate}\itemsep0mm
\item[(ii)] If $d(H,\chi) = 0$ and $\chi(H) \geq \chi(G) - 1$, it means firstly, that $H$ admits a $J$-colouring and secondly, if $c(v) = c_i$, $v \in V(G)$ the vertices of the copy of $H$ which is corona'ed to $v$ can be coloured by $\{c_1,c_2,c_3,\ldots,v_{i-1},c_{i+1}, \ldots, c_{\ell+1}\}$. Therefore, $\chi(G\circ H)=\chi(H)=|c(N[u])|=d_\chi(u),\ \forall\, u \in V(G\circ H)$. Hence, $ren(G\circ H)=0$.

\item[(ii)] Since $H$ is $J$-colourable, $\sum\limits_{i=\ell+2}^{k}\theta(c_i)$ is the minimum number of vertices to be removed from $G$ together with the corresponding copies of $H$, each representing the removal of $m$ vertices, to result in a corona graph $G'\circ H$ which has $H$ both, $J$-colourable and has $\chi(H)=\chi(G')-1$. Therefore, the result from Part (i).
\end{enumerate}
\end{proof}


\begin{theorem}\label{Thm-3.5}
Consider graphs $G$ and $H$ of order $n \geq 2$ and $m \geq 2$, and $\chi(G) = k$, $\chi(H) = \ell$, respectively. Also let both graphs be chromatically coloured. Then, if $ren(H) \geq 1$, we have 
\begin{enumerate}
\item[(i)] $ren(G\circ H)=n\cdot ren(H)$, if after removal of $ren(H)$ vertices from all copies of $H$, $\chi(H')\geq k-1$.
\item[(ii)] $ren(G\circ H)=n\cdot ren(H)+(m+1)\cdot\sum\limits_{i=\chi(H')+2}^{k}\theta(c_i)$, if after removal of $ren(H)$ vertices from all copies of $H$, $\chi(H')<k-1$.
\end{enumerate}
\end{theorem}
\begin{proof}
It is obvious that by removing the appropriate $ren(H)$ vertices from each copy of $H$ to obtain $n$ copies of $H'$ each among which allows a $J$-colouring, the conditions of Lemma \ref{Lem-3.4} will apply. Therefore the result immediately follows.
\end{proof}

\begin{corollary}
Consider graphs $G$ and $H$ of order $n \geq 2$ and $m \geq 2$, and $\chi(G) = k$, $\chi(H) = \ell$, respectively. Also let both graphs be chromatically coloured. Then, we have 

\begin{enumerate}\itemsep0mm
\item[(i)] if $d(H,\chi)=0$ and $\ell\ge k-1$, then $d(G\circ H,\chi)=0$;
\item[(ii)] if $d(H,\chi)=0$ and $\ell< k-1$, then $d(G\circ H, \chi)=k-(\ell + 1)$;
\item[(iii)] if  after removal of $ren(H)$ vertices from all copies of $H$, $\chi(H')\geq k-1$, then $d(G\circ H,\chi)= d(H,\chi)$;
\item[(iv)] if after removal of $ren(H)$ vertices from  all copies of $H$, $\chi(H')<k-1$, then $d(G\circ H, \chi)=d(H,\chi)+ k-(\chi(H')+1)$. 
\end{enumerate}
\end{corollary}
\begin{proof}
The result is a direct consequence of the proofs of Lemma \ref{Lem-3.4} and Theorem \ref{Thm-3.5}.
\end{proof}

\subsection{Some Comments on Fundamental Principles}

Since $K_1$ is a connected graph, a vertex can considered as inherently adjacent to itself. Hence, in the edgeless graph on $n$ vertices denoted $\mathfrak{N}_n$, the conventional distance between two vertices $u$ and $v$ can inherently only be between vertices say, $v$ and $v$ itself and $d(v,v)=0$. Between two distinct vertices in $\mathfrak{N}_n$ it follows that $d(u,v)$ is undefined and in an infinite connected graph $G$ the diameter $diam(G)=\infty$.

In terms of chromatic degrees, let the chromatic distance between two vertices in graph $G$ be defined as, $d_\chi(u,v)=|d_\chi(u)-d_\chi(v)|$. Therefore, a connected graph of order $n \geq 2$ can have $d_\chi(u,v)=0$ for any two vertices. A complete graph $K_n$ and a path $P_n$ serve as easy examples. Such a graph is called a \textit{chromatic null graph}. It means that all graphs which allow a $J$-colouring belong to the family of chromatic null graphs. Furthermore, whereas $d(u)=d(v)=0$ in a null graph, $d(u)$ and $d(v)$ need not be equal to allow, $d_\chi(u,v)=0$. Since only $d_\chi(u)= d_\chi(v),\ \forall\, u,v \in V(G)$ is required it follows that, $d_\chi(u), d_\chi(v) \in \N$.

\section{Conclusion}

We note that for if $C_{nm}$ admits a $J$-colouring and the Jahangir graph $J_{n,m}$ does not, the removal of vertex $u$ is sufficient to render an induced subgraph which admits a $J$-colouring. Hence, typically for Jahangir graphs the vertex with degree $\Delta(G)$ is removed. For a wheel graph $W_n$ which does not admit a $J$-colouring, a vertex typically with degree $\delta(G)$ is removed to obtain the result. It therefore appears that characterising the degree sequence of a graph will not be of assistance to find an efficient algorithm to determine $ren(G)$ in general.

Recall  the definition of the rainbow neighbourhood number of a graph $G$ as defined in \cite{6}.

\begin{definition}{\rm \cite{6} 
Let $G$ be a graph with a chromatic colouring $\mathcal{C}$ defined on it. The number of vertices in $G$ yielding rainbow neighbourhoods is called the \textit{rainbow neighbourhood number} of the graph $G$, denoted by $r_\chi(G)$.
}\end{definition}

If it exists, finding a relation between $ren(G)$, the order $n$ of a graph and $r_\chi(G)$ remains open. The various graph operations both, on and between graphs offer a wide scope for further research.

For this problem, revert to considering $\chi^-$-colourings. It is well known that a finite integer sequence $\mathsf{d}=(d_1,d_2,d_3,\ldots,d_n)$  is graphical if there exists a simple graph $G$ with degree sequence $\mathsf{d}$. Motivated by this fact and the importance thereof, the concept of a finite \textit{chromatic degree sequence}, denoted by $\mathsf{d}_\chi$, in respect of a $\chi^-$-colouring is introduced. As with a degree sequence, the finite number of vertices is prescribed by the finite length of the sequence. The problem to be answered is, for which finite integer sequences does a corresponding simple graph exist such that a sequence $\mathsf{d}$ is a chromatic degree sequence, $\mathsf{d}_\chi$. See Definition \ref{Def-3.1} for chromatic degree $d_\chi(v)$, $v \in V(G)$. Such a sequence is called a \textit{chromatically graphic sequence}.
\vspace{0.25cm}

\textbf{Dedication:} The first author wishes to dedicate this research paper to Bob Dylan who endlessly through his songs, ties multiple knots in the mind and soul of the first author. Unknotting takes much thinking about the wisdom Bob Dylan shares with the world.


\end{document}